\documentclass[12pt,english,reqno]{amsart}
\usepackage{a4wide}
\usepackage{amsthm,amsmath,amssymb,amsfonts}

\usepackage[
  pdfauthor={Bernd C. Kellner},
  pdfkeywords={Riemann zeta values, Eisenstein polynomial, Stirling numbers},
  pdftitle={On quotients of Riemann zeta values at odd and even integer arguments},
  linktocpage,colorlinks,bookmarksnumbered]{hyperref}

\title[On quotients of Riemann zeta values]
{On quotients of Riemann zeta values\\at odd and even integer arguments}
\author{Bernd C. Kellner}
\subjclass[2010]{11M06, 11R09 (Primary) 11B73 (Secondary)}
\email{bk@bernoulli.org}
\keywords{Riemann zeta values, Eisenstein polynomial, Stirling numbers}

\newcommand{\RR}{\mathbb{R}}
\newcommand{\CC}{\mathbb{C}}
\newcommand{\QQ}{\mathbb{Q}}
\newcommand{\ZZ}{\mathbb{Z}}
\newcommand{\NN}{\mathbb{N}}

\newcommand{\SE}{\mathcal{E}}
\newcommand{\HF}{\mathcal{H}}
\newcommand{\HN}{\mathbf{H}}
\newcommand{\LF}{\mathcal{L}}
\newcommand{\LFS}{\mathcal{L}^\star}
\newcommand{\ST}{\mathbf{S}}
\newcommand{\SF}[2]{\genfrac{\langle}{\rangle}{0pt}{}{#1}{#2}}

\newcommand{\fs}{\mathbf{F}}
\newcommand{\fh}{\widehat{\mathbf{F}}}

\newcommand{\pp}{\mathfrak{p}}
\newcommand{\qq}{\mathfrak{q}}

\DeclareMathOperator{\re}{Re}
\DeclareMathOperator{\ord}{ord}

\newcommand{\pdiv}{\mid}
\newcommand{\Dop}{\nabla}
\newcommand{\valueat}[1]{{\,}_{\big|\, #1}}
\newcommand{\coeff}[1]{\left[#1\right]}

\newtheorem{prop}{Proposition}[section]
\newtheorem{lemma}[prop]{Lemma}
\newtheorem{theorem}[prop]{Theorem}
\newtheorem{conj}[prop]{Conjecture}

\theoremstyle{definition}
\newtheorem{tbl}[prop]{Table}

\numberwithin{equation}{section}

\begin{document}

\begin{abstract}
We show for even positive integers $n$ that the quotient of the Riemann
zeta values $\zeta( n+1 )$ and $\zeta( n )$ satisfies the equation
\[
  \frac{\zeta( n+1 )}{\zeta( n )} = \left( 1-\frac{1}{n} \right)
    \left( 1-\frac{1}{2^{n+1}-1} \right)
    \frac{\mathcal{L}^\star(\mathfrak{p}_n)}{\mathfrak{p}_n'(0)},
\]
where $\mathfrak{p}_n \in \mathbb{Z}[x]$ is a certain monic polynomial of
degree $n$ and $\mathcal{L}^\star: \mathbb{C}[x] \to \mathbb{C}$ is a linear
functional, which is connected with a special Dirichlet series. There exists the
decomposition $\mathfrak{p}_n(x) = x(x+1) \mathfrak{q}_n(x)$. If $n = p+1$
where $p$ is an odd prime, then $\mathfrak{q}_n$ is an Eisenstein polynomial
and therefore irreducible over $\mathbb{Z}[x]$.
\end{abstract}

\maketitle

\section{Introduction}

Euler's beautiful formula
\begin{equation} \label{eq:zeta-euler}
  \zeta(n) = -\frac{1}{2} \frac{(2\pi i)^n}{n!} B_n \quad (n \in 2\NN)
\end{equation}
establishes a relationship between the Riemann zeta function
\[
  \zeta(s) = \sum_{n \geq 1} n^{-s} \quad (s \in \CC, \; \re(s) > 1)
\]
at even positive integer arguments and the Bernoulli numbers $B_n$ defined by
\[
  \frac{s}{e^s-1} = \sum_{n \geq 0} B_n \frac{s^n}{n!} \quad (|s| < 2 \pi).
\]
These numbers are rational where $B_n = 0$ for odd $n > 1$.
The functional equation (cf.~\cite{Patterson:1995})
\begin{equation} \label{eq:zeta-func-eq}
  \zeta(1-s) = 2 (2\pi)^{-s} \cos\left( \frac{\pi s}{2} \right)
    \Gamma(s)\zeta(s) \quad (s \in \CC)
\end{equation}
leads to the values at negative integer arguments given by
\begin{equation} \label{eq:zeta-1-n}
  \zeta( 1-n ) = - \frac{ B_{n} }{n} \quad (n \in \NN, \; n \geq 2),
\end{equation}
which have remarkable $p$-adic properties. Sequences of these divided Bernoulli
numbers in certain arithmetic progressions encode information about zeros of
nontrivial $p$-adic zeta functions; so far only unique simple zeros have been
found, see \cite{Kellner:2007} and \cite{Koblitz:1996} for the theory.
\medskip

It remains an open question whether there exists a closed formula for
$\zeta(3), \zeta(5), \zeta(7), \ldots$ with ``certain" constants in the sense
of \eqref{eq:zeta-euler}. However, our goal is to show some properties of the
quotients $\zeta( n+1 ) / \zeta( n )$ for $n \in 2\NN$ relating them to a
``logarithmic part".
\medskip

Define the linear forward difference operator $\Dop$ and its powers by
\[
  \Dop^n f(x) = \sum_{\nu=0}^n (-1)^{n-\nu} \binom{n}{\nu} f(x+\nu)
\]
for integers $n \geq 0$ and any function $f: \CC \to \CC$.
We use the expression, for example,
\[
  \Dop^n f(x+y) \valueat{x = 1}
\]
to indicate the variable and an initial value when needed.
\medskip

In 1930 Hasse \cite{Hasse:1930} constructed a globally convergent series
representation of $\zeta$. He further showed that the following series
representation of the Dirichlet eta function, derived by Knopp via Euler
transformation \cite{Knopp:1922},
\begin{equation} \label{eq:eta-hasse}
  \eta(s) = ( 1 - 2^{1-s} ) \zeta(s) = \sum_{n \geq 0} \frac{(-1)^n}{2^{n+1}}
    \Dop^n x^{-s} \valueat{x = 1}
\end{equation}
is also valid for all $s \in \CC$ and uniformly convergent on any compact
subset. Compared to the first series of Hasse mentioned above, this gives
a globally and faster convergent series of $\zeta$ except for the set
\[
  \SE = \{ 1 + 2\pi i \, n / \log 2 : n \in \ZZ \}.
\]
The derivative of \eqref{eq:eta-hasse} leads to the function,
see \cite{Guillera&Sondow:2008} for a wider context,
\begin{equation} \label{eq:h-def}
  \HF(s) = 2 \eta'(s)
    = \sum_{n \geq 1} \frac{(-1)^{n+1}}{2^n} \Dop^n x^{-s} \log x
    \valueat{x=1} \quad (s \in \CC),
\end{equation}
which satisfies
\begin{equation} \label{eq:h-zeta}
  \HF(s) = 2^{2-s} \zeta(s) \log 2 + 2(1-2^{1-s}) \zeta'(s) \quad (s \in \CC).
\end{equation}

For now, let $n \in 2\NN$. Note that $n \notin \SE$. Considering the functional
equation \eqref{eq:zeta-func-eq} and using \eqref{eq:h-zeta}, one easily sees that
\[
  \zeta( n+1 ) = 2 \frac{(2\pi i)^n}{n!} \zeta'(-n)
    = - \frac{(2\pi i)^n}{n!} \frac{\HF(-n)}{2^{n+1}-1},
\]
since $\zeta(-n)=0$ is a trivial zero by \eqref{eq:zeta-1-n}.
By Euler's formula \eqref{eq:zeta-euler} we finally derive that
\begin{equation} \label{eq:zeta-quot}
  \frac{\zeta( n+1 )}{\zeta( n )} = \frac{2}{B_n} \frac{\HF(-n)}{2^{n+1}-1}.
\end{equation}

We will consider a special Dirichlet series that is connected with \eqref{eq:h-def}.
Let $o(\cdot)$ be Landau's little $o$ symbol.

\begin{theorem} \label{thm:l-func}
The Dirichlet series
\[
  \LF(s) = \sum_{n \geq 1} a_n n^{-s} \quad (s \in \CC)
\]
is an entire function on $\CC$ where
\[
  a_n = \frac{(-1)^{n+1}}{2^n} \Dop^n \log x \valueat{x = 1} = o(2^{-n}).
\]
More precisely, the values
\[
  l_n = (-1)^{n+1} \Dop^n \log x \valueat{x = 1} \in (0,\log 2]
    \quad (n \geq 1)
\]
define a strictly decreasing sequence with limit $0$.
The linear functional
\[
  \LFS: \CC[x] \to \CC, \quad  \LFS(f) = \sum_{n \geq 1} a_n f(n),
\]
is absolutely convergent for any $f \in \CC[x]$. A special value is given by
\[
  \LF(0)=\LFS(1)=\log(\pi/2).
\]
\end{theorem}

\begin{theorem} \label{thm:p-poly}
We have the following relations for $n \geq 1$:
\[
  \HF(-n) = 2^{-n} \LFS( \pp_n )
\]
where
\begin{align*}
  \pp_n(x) &= \sum_{\nu=0}^n (-1)^\nu 2^{n-\nu} \binom{x+\nu}{\nu}
    \Dop^\nu (x+1)^n, \\
  \pp_1(x) &= x + 1, \\
  \pp_2(x) &= x(x+1), \\
  \pp_n(x) &= \begin{cases}
       (x+1)^2 \, \qq_n(x) & \text{if odd} \; n \geq 3, \\
       x(x+1)  \, \qq_n(x) & \text{if even} \; n \geq 4.
    \end{cases}
\end{align*}
The polynomials $\pp_n, \qq_n \in \ZZ[x]$ are monic of degrees $n$, $n-2$,
respectively. The value $\pp_n(0)$ is related to the tangent numbers by
\[
  \pp_n(0) = \tanh^{(n)}(0) = 2^{n+1} ( 2^{n+1} - 1 ) \frac{B_{n+1}}{n+1}
    \quad (n \geq 1).
\]
Special values of the derivative are given by
\begin{alignat*}{3}
  \pp_n'(-1) &= - \pp_{n-1}(0) && (n \geq 2), \\
  \pp_n'(0)  &= (n-1) \, \pp_{n-1}(0) & \quad & (n \in 2\NN).
\end{alignat*}
\end{theorem}
\medskip

Now, we can state our main result. We further use the notations $\pp_n$
and $\qq_n$ below.

\begin{theorem} \label{thm:quot-zeta}
If $n \in 2\NN$, then
\begin{equation} \label{eq:quot-zeta}
  \frac{\zeta( n+1 )}{\zeta( n )} = \left( 1-\frac{1}{n} \right)
    \left( 1-\frac{1}{2^{n+1}-1} \right) \frac{\LFS( \pp_n )}{\pp_n'(0)}.
\end{equation}
There exists the limit
\[
  \lim_{\substack{n \to \infty\\2 \pdiv n}} \frac{\LFS( \pp_n )}{\pp_n'(0)} = 1.
\]
Generally, if $\alpha_j \in \QQ^\times$ and $n_j \in 2\NN$ for $j=1,\ldots,N$
and fixed $N \geq 1$, where the integers $n_j$ are strictly increasing, then
there exists $\pp \in \QQ[x]$ with $\deg \pp = n_N$ such that
\[
  \sum_{j=1}^N \alpha_j \, \zeta( n_j+1 ) / \zeta( n_j ) = \LFS( \pp ).
\]
\end{theorem}

\begin{tbl}
\begin{alignat*}{5}
  \zeta( 3 ) / \zeta( 2 ) &= &\frac{3}{7} \, \LFS(\pp_2),
    & \quad & \pp_2(x) &= x^2 + x, \\
  \zeta( 5 ) / \zeta( 4 ) &= &-\frac{15}{124} \, \LFS(\pp_4),
    && \pp_4(x) &= x^4 - 2 x^3 - 9 x^2 - 6 x \\
    &&&&&= (x^2 + x)( x^2 - 3x -6 ), \\
  \zeta( 7 ) / \zeta( 6 ) &= &\frac{21}{2032} \, \LFS(\pp_6),
    && \pp_6(x) &= x^6 - 9 x^5 - 15 x^4 + 65 x^3 + 150 x^2 + 80 x \\
    &&&&&= (x^2 + x)(x^4 - 10 x^3 - 5 x^2 + 70 x + 80 ).
\end{alignat*}
\end{tbl}

Of course one can simplify \eqref{eq:quot-zeta} to
\[
  \zeta( n+1 )/\zeta( n ) = ( 2^{n-1} ( 2^{n+1} - 1 ) B_n )^{-1} \,
    \LFS( \pp_n ) \quad (n \in 2\NN),
\]
but this needs again the definition of the Bernoulli numbers.

\begin{theorem} \label{thm:qn-irr}
Let $p$ be an odd prime and $n = p+1$. Then the polynomial $\qq_n$
is an Eisenstein polynomial and consequently irreducible over $\ZZ[x]$.
\end{theorem}

On the basis of some computations, see also Table~\ref{tbl:comp-q},
we raise the following conjecture.

\begin{conj}
The polynomials $\qq_n$ are irreducible over $\ZZ[x]$ for all $n \geq 4$.
\end{conj}

\section{Preliminaries}

We need some lemmas to prove the theorems in the following. For properties of
binomial coefficients, Stirling numbers, and finite differences we refer to
\cite{Graham&others:1994}.
\medskip

The harmonic numbers are defined by
\[
  \HN_n = \sum_{k=1}^n \frac{1}{k} \quad (n \geq 1).
\]
The Stirling numbers $\ST_1(n,k)$ of the first kind and $\ST_2(n,k)$ of the
second kind are usually defined by
\begin{alignat}{3}
 (x)_n &= \sum_{k=0}^n \ST_1(n,k) x^k   && \quad (n \geq 0), \label{eq:s1-def} \\
  x^n  &= \sum_{k=0}^n \ST_2(n,k) (x)_k && \quad (n \geq 0), \label{eq:s2-def}
\end{alignat}
where the falling factorials are given by
\[
  (x)_0 = 1 \quad \text{and} \quad (x)_k = x(x-1) \cdots (x-k+1)
    \quad \text{for} \quad k \geq 1.
\]
We further use the related numbers
\begin{equation} \label{eq:sf-def}
  \SF{n}{k} = \Dop^k x^n \valueat{x=0} = k! \, \ST_2(n,k),
\end{equation}
which obey the recurrence
\begin{equation} \label{eq:sf-rec}
  \SF{n+1}{k} = k \left( \SF{n}{k} + \SF{n}{k-1} \right).
\end{equation}
Then \eqref{eq:s2-def} turns into
\begin{equation} \label{eq:sf-binom}
  x^n = \sum_{k=0}^n \SF{n}{k} \binom{x}{k} \quad (n \geq 0).
\end{equation}
Note that $\ST_1(n,n)=\ST_2(n,n)=\ST_2(n,1)=1$ and $\ST_1(n,1)=(-1)^{n-1}(n-1)!$
for $n \geq 1$. Further $\ST_1(n,0)=\ST_2(n,0)=\SF{n}{0}=\delta_{n,0}$
for $n \geq 0$ using Kronecker's delta.

\begin{lemma} \label{lem:fac-p-1}
If $p$ is a prime and $k \geq 0$, then
\[
  (1+(x-1)_{p-1})^{(k)} \valueat{x = 0} \equiv -\delta_{k,p-1} \pmod{p}.
\]
\end{lemma}

\begin{proof}
The case $p=2$ is trivial. Let $p$ be an odd prime. 
By \eqref{eq:s1-def} we obtain that
\[
  1 + (x-1)_{p-1} = 1 + (x)_p/x = 1 + \sum_{k=1}^p \ST_1(p,k) x^{k-1}
    \equiv x^{p-1} \pmod{p}
\]
using the property (\cite[6.51, p.~314]{Graham&others:1994})
\[
  \ST_1(p,k) \equiv 0 \pmod{p} \quad (1 < k < p)
\]
and Wilson's theorem to derive that
$\ST_1(p,1)=(-1)^{p-1}(p-1)! \equiv -1 \pmod{p}$.
Only the $(p-1)$th derivative of $x^{p-1}$ at $x=0$
does not vanish and equals $-1 \pmod{p}$ using Wilson's theorem again.
\end{proof}

\begin{lemma} \label{lem:diff-xn}
Let $k, n \geq 1$ be integers. Then
\[
  \Dop^k x^n \valueat{x = 1} = \SF{n}{k} + \SF{n}{k+1}
    = \frac{1}{k+1} \Dop^{k+1} x^{n+1} \valueat{x = 0}.
\]
\end{lemma}

\begin{proof}
This follows by
\[
  \Dop^k x^n \valueat{x = 1} = \Dop^k x^n \valueat{x = 0}
    + \Dop^{k+1} x^n \valueat{x = 0}
\]
and using \eqref{eq:sf-def} and \eqref{eq:sf-rec}.
\end{proof}

\begin{lemma} \label{lem:diff-xp}
If $p$ is a prime and $k \geq 0$, then
\[
  \frac{1}{k!} \Dop^k x^p \valueat{x = 1} \equiv d_{p,k} \pmod{p}
\]
where $d_{p,k} = 1$ for $k=0,1,p$ and $d_{p,k}=0$ otherwise.
\end{lemma}

\begin{proof}
By Lemma~\ref{lem:diff-xn} and \eqref{eq:sf-def} we have
\[
  \tilde{d}_{p,k} := \frac{1}{k!} \Dop^k x^p \valueat{x = 1}
    = \frac{1}{k!} \left( \SF{p}{k} + \SF{p}{k+1} \right)
    = \ST_2(p,k) + (k+1) \ST_2(p,k+1).
\]
It is well known (\cite[6.51, p.~314]{Graham&others:1994}) that
\[
  \ST_2(p,k) \equiv 0 \pmod{p} \quad (1 < k < p).
\]
Thus $\tilde{d}_{p,k} \equiv d_{p,k} \pmod{p}$ follows easily.
\end{proof}

\begin{lemma}[Leibniz rule] \label{lem:leibniz}
For integers $n \geq 0$ and any functions $f, g: \CC \to \CC$, we have
\[
  \Dop^n (f(s)g(s)) = \sum_{\nu=0}^n \binom{n}{\nu}
    (\Dop^\nu f(s))(\Dop^{n-\nu} g(s+\nu)).
\]
\end{lemma}

\begin{lemma} \label{lem:diff-xr-log}
Let $k, n, r \geq 1$ be integers and $f: \NN \to \CC$ be a function. We have
\[
  \Dop^n x^r f(x) \valueat{x = 1} = \sum_{\nu=0}^{\min(n,r)}
    \lambda_{r,\nu}(n) \Dop^{n-\nu} f(x) \valueat{x = 1}
\]
with
\[
  \lambda_{r,\nu}(n) = \sum_{k=\nu}^r \SF{r}{k} \binom{n}{\nu}
    \binom{n+1-\nu}{k-\nu}
\]
where $\lambda_{r,\nu} \in \ZZ[x]$ and $\deg \lambda_{r,\nu} = r$.
\end{lemma}

\begin{proof}
We use Lemma~\ref{lem:leibniz} to obtain that
\[
  \Dop^n x^r f(x) \valueat{x = 1}
    = \sum_{\nu=0}^{\min(n,r)} \binom{n}{\nu} \Dop^\nu (x+n-\nu)^r
      \valueat{x = 1} \Dop^{n-\nu} f(x) \valueat{x = 1}.
\]
Applying $\Dop^\nu$ to \eqref{eq:sf-binom} provides that
\begin{align}
  \lambda_{r,\nu}(n)
    &= \binom{n}{\nu} \Dop^\nu (x+n-\nu)^r \valueat{x = 1} \nonumber \\
    &= \sum_{k=\nu}^r \SF{r}{k} \binom{n}{\nu}
       \binom{n+1-\nu}{k-\nu} \label{eq:lambda-sf} \\
    &= \sum_{k=\nu}^r \binom{k}{\nu} \ST_2(r,k) (n)_{\nu} (n+1-\nu)_{k-\nu}.
    \label{eq:lambda-s2}
\end{align}
The last equation follows by \eqref{eq:sf-def} and implies that
$\lambda_{r,\nu} \in \ZZ[x]$ and $\deg \lambda_{r,\nu} = r$.
\end{proof}

\begin{prop}[{\cite[Theorem~1.3]{Kellner:2012}}] \label{prop:fh-ident}
Define
\[
  \fs_n(x) = \sum_{\nu=1}^n \SF{n}{\nu} x^\nu, \quad
    \fh_n(x) = \sum_{\nu=1}^n \SF{n}{\nu} \HN_\nu \, x^\nu \quad (n \geq 1).
\]
For $n \in 2\NN$ we have the identity
\[
  \fh_n \! \left( \! -\frac{1}{2} \right) =
    - \frac{n-1}{2} \, \fs_{n-1} \! \left( \! -\frac{1}{2} \right).
\]
\end{prop}

\begin{prop}[Kummer congruences {\cite[Theorem~7 (2), p.~44]{Koblitz:1996}}]
\label{prop:kummer-congr}
Let $p > 3$ be a prime and $n, m \in 2\NN$. 
If $n \equiv m \not\equiv 0 \pmod{p-1}$, then
\[
  \frac{B_n}{n} \equiv \frac{B_m}{m} \pmod{p}.
\]
\end{prop}

\section{Proofs}

\begin{proof}[Proof of Theorem \ref{thm:l-func}]
Let $n \geq 1$. We use the identity, cf.~\cite[p.~54]{Norlund:1924},
\[
  \Dop^n \log x \valueat{x = a} = (-1)^{n-1} (n-1)! \int_a^{a+1}
    \frac{dt}{t(t+1)\cdots (t+n-1)}.
\]
Hence
\[
  l_n = (-1)^{n+1} \Dop^n \log x \valueat{x = 1}
    = \frac{1}{n} \int_0^1 \frac{dt}{(t+1)(\tfrac{t}{2}+1)
      \cdots(\tfrac{t}{n}+1)}.
\]
Define $\phi_n(t) = 1 + t/n$ where $\phi_n$ maps $(0,1]$ onto $(1,1+1/n]$.
The integrand above differs by the factor $\phi_{n+1}^{-1}$ regarding $n$
and $n+1$. Estimating these integrals we then obtain that
\[
  n l_n > (n+1) l_{n+1} > n l_{n+1}.
\]
Consequently $(l_n)_{n \geq 1}$ defines a positive strictly decreasing sequence
\[
  \log 2 = l_1 > l_2 > l_3 > \ldots
\]
with limit $l_\infty = 0$. Thus, the coefficients of $\LF(s)$ obey that 
$a_n = o(2^{-n})$. As usual, write $s=\sigma +it$. The estimate
\[
  |\LF(s)| \leq \sum_{n \geq 1} \left| a_n n^{-s} \right|
    < \sum_{n \geq 1} 2^{-n} n^{-\sigma} < \infty
\]
for any $\sigma, t \in \RR$ shows that $\LF(s)$ is an entire function on $\CC$.
Since $\LFS$ is linear and $\LFS(x^m)=\LF(-m)$ for $m \in \NN_0$, it follows
that $\LFS$ is absolutely convergent for any $f \in \CC[x]$. The well-known
values $\zeta(0) = -\frac{1}{2}$ and $\zeta'(0) = -\frac{1}{2} \log ( 2\pi)$
yield the special value $\LFS(1)=\LF(0)=\HF(0)=\log(\pi/2)$ concerning
\eqref{eq:h-def} and \eqref{eq:h-zeta}.
\end{proof}

\begin{prop} \label{prop:eq-h-l}
We have the following relations for $r \geq 1$:
\[
  \HF(-r) = 2^{-r} \LFS( \pp_r )
\]
where
\begin{equation} \label{eq:pp-def}
  \pp_r(x) = \sum_{\nu=0}^r (-1)^\nu 2^{r-\nu} \lambda_{r,\nu}(x + \nu)
\end{equation}
with $\lambda_{r,\nu}$ as defined in Lemma~\ref{lem:diff-xr-log}.
The polynomial $\pp_r \in \ZZ[x]$ is monic of degree $r$.
\end{prop}

\begin{proof}
Let $r \geq 1$ be a fixed integer. By Lemma~\ref{lem:diff-xr-log} we know that
$\pp_r \in \ZZ[x]$ and $\deg \pp_r \leq r$, since $\lambda_{r,\nu} \in \ZZ[x]$
and $\deg \lambda_{r,\nu}=r$ for $\nu=0,\ldots,r$.
Due to absolute convergence ensured by Theorem~\ref{thm:l-func}, we derive that
\begin{align*}
  2^{-r} \LFS(\pp_r)
    &= \sum_{\nu=0}^r \sum_{n \geq 1} \frac{(-1)^{n+\nu+1}}{2^{n+\nu}}
       \lambda_{r,\nu}(n + \nu) \Dop^n \log x \valueat{x = 1} \\
    &= \sum_{\nu=0}^r \sum_{n > \nu} \frac{(-1)^{n+1}}{2^n}
       \lambda_{r,\nu}(n) \Dop^{n-\nu} \log x \valueat{x = 1} \\
    &= \sum_{n \geq 1} \frac{(-1)^{n+1}}{2^n} \sum_{\nu=0}^{\min(n,r)}
       \lambda_{r,\nu}(n) \Dop^{n-\nu} \log x \valueat{x = 1} \\
    &= \sum_{n \geq 1} \frac{(-1)^{n+1}}{2^n} \Dop^n x^r \log x \valueat{x=1} \\
    &= \HF(-r).
\end{align*}
The last two steps follow by Lemma~\ref{lem:diff-xr-log} and \eqref{eq:h-def}.
Define the linear functional
\[
  [ \,\cdot\, ]_n : \RR[x] \to \RR, \quad f \mapsto \frac{f^{(n)}(0)}{n!},
\]
giving the $n$th coefficient of a polynomial $f$. By virtue of
\eqref{eq:lambda-s2} and \eqref{eq:pp-def} we easily obtain
\begin{align*}
  \coeff{\pp_r(x)}_r
    &= \coeff{\sum_{\nu=0}^r (-1)^\nu 2^{r-\nu} \binom{r}{\nu}
       \ST_2(r,r) (x+\nu)_{r} (x+1)_0}_r \\
    &= \sum_{\nu=0}^r \binom{r}{\nu}  (-1)^\nu 2^{r-\nu} = (2-1)^r = 1,
\end{align*}
since all other terms vanish. This shows that $\pp_r$ is a monic polynomial
of degree $r$.
\end{proof}

\begin{prop} \label{prop:p-poly}
The polynomials $\pp_n$ have the following properties for $n \geq 1$:
\begin{align*}
  \pp_n(x) &= \sum_{\nu=0}^n (-1)^\nu 2^{n-\nu} \binom{x+\nu}{\nu}
              \Dop^\nu (x+1)^n, \\
  \pp_1(x) &= x + 1, \\
  \pp_2(x) &= x(x+1), \\
  \pp_n(x) &= \begin{cases}
       (x+1)^2 \, \qq_n(x) & \text{if odd} \; n \geq 3, \\
       x(x+1)  \, \qq_n(x) & \text{if even} \; n \geq 4. \\
    \end{cases}
\end{align*}
The polynomials $\qq_n \in \ZZ[x]$ are monic of degree $n-2$. Moreover,
\[
  \pp_n(0) = \tanh^{(n)}(0) = -\pp_{n+1}'(-1) = 2^{n+1} (2^{n+1} - 1)
    \frac{B_{n+1}}{n+1}.
\]
\end{prop}

\begin{proof}
Let $n \geq 1$. From \eqref{eq:lambda-sf} and \eqref{eq:pp-def} we deduce that
\begin{align}
  \pp_n(x)
    &= \sum_{\nu=0}^n (-1)^\nu 2^{n-\nu}
       \sum_{k=\nu}^n \SF{n}{k} \binom{x+\nu}{\nu} \binom{x+1}{k-\nu}
  \nonumber \\
    &= \sum_{\nu=0}^n (-1)^\nu 2^{n-\nu} \binom{x+\nu}{\nu} \Dop^\nu (x+1)^n,
  \label{eq:loc-pn-def}
\end{align}
where the last part follows by \eqref{eq:sf-binom} and applying $\Dop^\nu$.
Comparing with \eqref{eq:eta-hasse} and using \eqref{eq:zeta-1-n}
we instantly derive that
\begin{equation} \label{eq:loc-pn-0}
  \pp_n(0) = \sum_{\nu=0}^n (-1)^\nu 2^{n-\nu} \Dop^\nu x^n \valueat{x = 1}
    = - c_{n+1} \zeta( -n )
    = c_{n+1} \frac{B_{n+1}}{n+1}
\end{equation}
with an additional factor $c_{n+1} = 2^{n+1} ( 2^{n+1} - 1 )$. This equals the
$n$th tangent number except for the sign such that $\pp_n(0) = \tanh^{(n)}(0)$,
see \cite[p.~287]{Graham&others:1994}. We further obtain that
\begin{equation} \label{eq:loc-pn-1}
  \pp_n(-1) = \sum_{\nu=0}^n (-1)^\nu 2^{n-\nu} \binom{\nu-1}{\nu}
      \Dop^\nu x^n \valueat{x = 0}
    = 2^n \Dop^0 x^n \valueat{x = 0} = 0.
\end{equation}
Next, we show that
\begin{equation} \label{eq:loc-pn-2}
  \pp_{n+1}'(-1) = -\pp_n(0).
\end{equation}
Since $\pp_{n+1}(-1)=0$, we have $x+1$ as a factor of $\pp_{n+1}(x)$. Hence
\[
  \frac{\pp_{n+1}(x)}{x+1} = 2^{n+1} (x+1)^n +
    \sum_{\nu=1}^{n+1} (-1)^\nu 2^{n+1-\nu}
    \frac{1}{\nu} \binom{x+\nu}{\nu-1} \Dop^\nu (x+1)^{n+1}.
\]
By L'H\^{o}pital's rule we obtain that
\[
  \pp_{n+1}'(-1) = \lim_{x \to -1} \frac{\pp_{n+1}(x)}{x+1}
    = -\sum_{\nu=0}^n (-1)^\nu 2^{n-\nu}
    \frac{1}{\nu+1} \Dop^{\nu+1} x^{n+1} \valueat{x = 0}.
\]
In view of Lemma~\ref{lem:diff-xn} and \eqref{eq:loc-pn-0} we then get
\eqref{eq:loc-pn-2}. From \eqref{eq:loc-pn-0}, \eqref{eq:loc-pn-1},
\eqref{eq:loc-pn-2}, and $B_n=0$ for odd $n \geq 3$, we finally conclude that
$x+1 \pdiv \pp_n(x)$ for $n \geq 1$, $x \pdiv \pp_n(x)$ for even $n \geq 2$,
and $(x+1)^2 \pdiv \pp_n(x)$ for odd $n \geq 3$. The rest follows by the fact
that $\pp_n$ is a monic polynomial of degree $n$.
\end{proof}

\begin{prop} \label{prop:p0-deriv}
We have
\[
  \pp_n'(0) = (n-1) \, \pp_{n-1}(0) \quad (n \in 2\NN).
\]
\end{prop}

\begin{proof}
We first observe by \eqref{eq:loc-pn-0} and Lemma~\ref{lem:diff-xn} that
\begin{equation} \label{eq:loc-pn-0-sf}
  \pp_n(0) = \sum_{\nu=0}^n (-1)^\nu 2^{n-\nu}
      \left( \SF{n}{\nu} + \SF{n}{\nu+1} \right)
    = - \sum_{\nu=1}^n (-1)^\nu 2^{n-\nu} \SF{n}{\nu}
\end{equation}
for $n \geq 1$. Now, let $n \in 2\NN$. Define
\[
  H_k(x) = \sum_{j=1}^k \frac{1}{x+j} \quad (k \geq 1)
\]
where $\HN_k = H_k(0)$. Via \eqref{eq:loc-pn-def} one easily sees that 
the derivative is given by
\[
  \pp_n'(x) = 2n \, \pp_{n-1}(x) + f_n(x)
\]
with
\[
  f_n(x) = \sum_{\nu=1}^n (-1)^\nu 2^{n-\nu} \binom{x+\nu}{\nu} H_\nu(x)
    \Dop^\nu (x+1)^n.
\]
Define
\[
  h_n = \sum_{\nu=1}^n (-1)^\nu 2^{n-\nu} \HN_\nu \SF{n}{\nu}.
\]
Using Lemma~\ref{lem:diff-xn} and \eqref{eq:loc-pn-0} we then obtain that
\begin{align*}
  f_n(0) &= \sum_{\nu=1}^n (-1)^\nu 2^{n-\nu} \HN_\nu
       \left( \SF{n}{\nu} + \SF{n}{\nu+1} \right) \\
    &= h_n - 2(h_n + \pp_{n-1}(0)) = -h_n - 2\pp_{n-1}(0),
\end{align*}
where the second part follows by the substitution
$\HN_\nu = \HN_{\nu+1} - \frac{1}{\nu+1}$ and \eqref{eq:loc-pn-0-sf}. Thus
\[
  \pp_n'(0) = 2(n-1) \, \pp_{n-1}(0) - h_n.
\]
With the help of Proposition~\ref{prop:fh-ident} we finally achieve that
\[
  h_n = 2^n \fh_n (-1/2) = -2^{n-1} (n-1) \fs_{n-1} (-1/2)
    = (n-1) \, \pp_{n-1}(0).
\]
The last equation follows by \eqref{eq:loc-pn-0-sf} and shows the result.
\end{proof}

\begin{proof}[Proof of Theorem \ref{thm:p-poly}]
The proof jointly follows by Propositions \ref{prop:eq-h-l},
\ref{prop:p-poly}, and \ref{prop:p0-deriv}.
\end{proof}

\begin{proof}[Proof of Theorem \ref{thm:quot-zeta}]
Let $n \in 2\NN$. By \eqref{eq:zeta-quot} and Proposition~\ref{prop:eq-h-l}
we have
\begin{equation} \label{eq:loc-zn-1}
  \zeta( n+1 )/\zeta( n ) = ( 2^{n-1} ( 2^{n+1} - 1 ) B_n )^{-1} \,
    \LFS( \pp_n ).
\end{equation}
From Propositions \ref{prop:p-poly} and \ref{prop:p0-deriv} we obtain that
\begin{equation} \label{eq:loc-zn-2}
  \pp_n'(0) = \frac{n-1}{n} 2^n ( 2^n - 1 ) B_n.
\end{equation}
Plugging \eqref{eq:loc-zn-2} into \eqref{eq:loc-zn-1} yields that
\[
  \frac{\zeta( n+1 )}{\zeta( n )} = \left( 1-\frac{1}{n} \right)
    \left( 1-\frac{1}{2^{n+1}-1} \right) \frac{\LFS( \pp_n )}{\pp_n'(0)}.
\]
Since $\zeta(n) \to 1$, $1 - 1/n \to 1$, and $1 - 1/(2^{n+1}-1) \to 1$ as
$n \to \infty$, we derive that
\[
  \lim_{\substack{n \to \infty\\2 \pdiv n}} \frac{\LFS( \pp_n )}{\pp_n'(0)}
    = 1.
\]
It remains the second part. Since $\LFS$ is linear, we obtain by
\eqref{eq:loc-zn-1} that
\begin{equation} \label{eq:loc-zn-3}
  \alpha \, \zeta( n+1 ) / \zeta( n ) = \LFS( \hat{\pp} )
\end{equation}
where $\alpha \in \QQ^\times$ and $\hat{\pp} \in \QQ[x]$ with $\deg \hat{\pp} = n$. 
Thus, the claimed sum, consisting of terms as in \eqref{eq:loc-zn-3}, follows by 
the assumption that the integers $n_j \in 2\NN$ are strictly increasing.
\end{proof}

\begin{lemma} \label{lem:pn-deriv}
Let $n > k \geq 1$ be integers. The $k$th derivative of $\pp_n$ is given by
\[
  \pp_n^{(k)}(x) = \sum_{j=0}^k \binom{k}{j} (n-1)_j
    \sum_{\nu=0}^{n-1-j} (-1)^\nu 2^{n-1-\nu} (x+\nu+1)_{\nu+1}^{(k-j)}
    \frac{1}{\nu!} \Dop^\nu (x+1)^{n-1-j}.
\]
\end{lemma}

\begin{proof}
By Lemma~\ref{lem:leibniz} we have the identity
\[
  \Dop^\nu (x+1)^n = \nu \Dop^{\nu-1} (x+1)^{n-1} + (x+\nu+1)
    \Dop^\nu (x+1)^{n-1} \quad (\nu \geq 1).
\]
From \eqref{eq:loc-pn-def} we then derive that
\begin{align*}
  \pp_n(x)
    &= \sum_{\nu=0}^n (-1)^\nu 2^{n-\nu} (x+\nu)_\nu \frac{1}{\nu!}
       \Dop^\nu (x+1)^n \\
    &= \sum_{\nu=0}^{n-1} (-1)^\nu 2^{n-1-\nu} (x+\nu+1)_{\nu+1}
       \frac{1}{\nu!} \Dop^\nu (x+1)^{n-1}
\end{align*}
using the above identity and after some index shifting.
The $k$th derivative follows by the Leibniz rule applied to the terms
$(x+\nu+1)_{\nu+1}$ and $\Dop^\nu (x+1)^{n-1}$.
\end{proof}

\begin{prop} \label{prop:p0-deriv-k}
If $n = p+1$ with $p$ an odd prime, then
\[
  \pp_n^{(k)}(0) \equiv 0 \pmod{p} \quad (1 \leq k \leq n-2).
\]
\end{prop}

\begin{proof}
Since $n-1 = p$, we obtain by Lemma~\ref{lem:pn-deriv} that
\begin{align*}
  \pp_n^{(k)}(0)
    &\equiv \sum_{\nu=0}^{n-1} (-1)^\nu 2^{n-1-\nu}
      \, (x+\nu+1)_{\nu+1}^{(k)} \valueat{x = 0}
      \, \frac{1}{\nu!} \Dop^\nu x^{n-1} \valueat{x = 1} \\
    &\equiv \sum_{\nu \in \{ 0,1,p \}} (-1)^\nu 2^{n-1-\nu}
      \, (x+\nu+1)_{\nu+1}^{(k)} \valueat{x = 0} \\
    &\equiv \left( 2 (x+1) - (x+2)_2
      - (x+p+1)_{p+1} \right)^{(k)} \! \valueat{x = 0} \pmod{p}
\end{align*}
applying Lemma~\ref{lem:diff-xp} and Fermat's little theorem.
Using the identities
\begin{align*}
  -(x^2+x) &= 2 (x+1) - (x+2)_2, \\
  (x+p+1)_{p+1} &\equiv (x+p-1)_{p-1}(x+p+1)_2 \equiv (x-1)_{p-1}(x^2+x)
    \pmod{p},
\end{align*}
we achieve that
\begin{align*}
  \pp_n^{(k)}(0)
    &\equiv -\left( A(x) B(x) \right)^{(k)} \! \valueat{x = 0} \\
    &\equiv -\sum_{j=1}^k \binom{k}{j} A(x)^{(j)} B(x)^{(k-j)}
      \! \valueat{x = 0} \pmod{p}
\end{align*}
where
\[
  A(x) = x^2+x, \quad B(x) = 1+(x-1)_{p-1}.
\]
By Lemma~\ref{lem:fac-p-1} we have
\[
  B(x)^{(k-j)} \! \valueat{x = 0} \equiv - \delta_{k-j,p-1} \pmod{p}.
\]
Since $j \geq 1$ and $k-j < n-2 = p-1$, it follows that
$\pp_n^{(k)}(0) \equiv 0 \pmod{p}$.
\end{proof}

\begin{proof}[Proof of Theorem \ref{thm:qn-irr}]
Let $p$ be an odd prime and $n = p+1$.
By Proposition~\ref{prop:p-poly} we have the decomposition
$\pp_n(x) = x(x+1) \qq_n(x)$ with
\begin{align*}
  \pp_n(x) &= x^n + \alpha_{n-1} x^{n-1} + \cdots + \alpha_1 x, \\
  \qq_n(x) &= x^{n-2} + \beta_{n-3} x^{n-3} + \cdots + \beta_0.
\end{align*}
Hence, we have the relations
\begin{align*}
  \beta_0 &= \alpha_1, \\
  \beta_1 &= \alpha_2 - \beta_0, \\
          & \;\; \vdots \\
  \beta_{n-3} &= \alpha_{n-2} - \beta_{n-4}.
\end{align*}
Since $n-2 < p$, we obtain by Proposition~\ref{prop:p0-deriv-k} that
\[
  \alpha_\nu \equiv \pp_n^{(\nu)}(0)/\nu! \equiv 0 \pmod{p}
    \quad (1 \leq \nu \leq n-2).
\]
It easily follows by induction that
$p \pdiv \alpha_\nu$ for $1 \leq \nu \leq n-2$ implies that
$p \pdiv \beta_\nu$ for $0 \leq \nu \leq n-3$.
Moreover, we derive by \eqref{eq:loc-zn-2} that
\[
  \beta_0 = \alpha_1 = \pp_n'(0) = (n-1) 2^n ( 2^n - 1 ) B_n/n.
\]
Now, we have to show that $\ord_p \beta_0 = 1$. It is well known that $B_2 = 1/6$ 
and $B_4 = -1/30$. For $p=3$ we compute $\beta_0 = -6$. Let $p \geq 5$.
By Kummer congruences via Proposition~\ref{prop:kummer-congr} and Fermat's little 
theorem we conclude that
\[
  2^n ( 2^n - 1 ) B_n/n \equiv 2^2 ( 2^2 -1 ) B_2/2 \equiv 1 \pmod{p}
\]
and further that $\ord_p \beta_0 = 1$, since $p = n-1$. Altogether, this shows 
that $\qq_n$ is an Eisenstein polynomial and therefore $\qq_n$ is irreducible 
over $\ZZ[x]$.
\end{proof}

\appendix

\section{Computations}

\begin{tbl} \label{tbl:comp-p}
Polynomials $\pp_n$:
\begin{align*}
  \pp_1(x) &= x + 1. \\
  \pp_2(x) &= x^2 + x. \\
  \pp_3(x) &= x^3 - 3 x - 2. \\
  \pp_4(x) &= x^4 - 2 x^3 - 9 x^2 - 6 x. \\
  \pp_5(x) &= x^5 - 5 x^4 - 15 x^3 + 5 x^2 + 30 x + 16. \\
  \pp_6(x) &= x^6 - 9 x^5 - 15 x^4 + 65 x^3 + 150 x^2 + 80 x. \\
  \pp_7(x) &= x^7 - 14 x^6 + 210 x^4 + 315 x^3 - 196 x^2 - 588 x - 272. \\ 
  \pp_8(x) &= x^8 - 20 x^7 + 42 x^6 + 448 x^5 + 105 x^4 - 2492 x^3 - 4116 x^2
              -1904 x.  
\end{align*}
\end{tbl}
\medskip

\begin{tbl} \label{tbl:comp-q}
Polynomials $\qq_n$:
\begin{align*}
  \qq_3(x) &= x - 2. \\
  \qq_4(x) &= x^2 - 3 x - 6. \\
  \qq_5(x) &= x^3 - 7 x^2 - 2 x + 16. \\
  \qq_6(x) &= x^4 - 10 x^3 - 5 x^2 + 70 x + 80. \\
  \qq_7(x) &= x^5 - 16 x^4 + 31 x^3 + 164 x^2 - 44 x - 272. \\
  \qq_8(x) &= x^6 - 21 x^5 + 63 x^4 + 385 x^3 - 280 x^2 - 2212 x - 1904. \\
  \qq_9(x) &= x^7 - 29 x^6 + 183 x^5 + 377 x^4 - 2512 x^3 - 5076 x^2 + 3088 x
              + 7936. \\
  \qq_{10}(x) &= x^8 - 36 x^7 + 306 x^6 + 504 x^5 - 7119 x^4 - 15204 x^3 
                 + 27804 x^2 + 99216 x + 71424.
\end{align*}
The polynomials $\qq_4, \ldots, \qq_{10}$ are irreducible over $\ZZ[x]$. 
\end{tbl}

\section*{Acknowledgment}

The author would like to thank the referee for valuable suggestions.

\bibliographystyle{amsplain}

\bigskip

\end{document}